\documentclass[11pt,twoside,letter]{article} 

\usepackage{amssymb}
\usepackage{graphicx}
\usepackage{amscd}
\usepackage{color}
\usepackage{float}
\usepackage{graphics}
\usepackage{amsmath}
\usepackage{amssymb}
\usepackage{mathrsfs}
\usepackage{amsthm}
\usepackage{amsfonts}
\usepackage{latexsym}
\usepackage{epsf}
\usepackage[]{hyperref}
\usepackage{fancyhdr}
\usepackage{mathptmx}
\usepackage{longtable}
\usepackage{array}
\usepackage{csquotes}
\usepackage{verbatim}
\usepackage{enumerate}
\usepackage[margin=1.5in]{geometry}
\usepackage{rotating}
\usepackage{caption}

\newcolumntype{P}[1]{>{\centering\arraybackslash}p{#1}}

\makeatletter
\newenvironment{proof*}[1][\proofname]{\par
  \pushQED{\qed}%
  \normalfont \partopsep=\z@skip \topsep=\z@skip
  \trivlist
  \item[\hskip\labelsep
        \itshape
    #1\@addpunct{.}]\ignorespaces
}{%
  \popQED\endtrivlist\@endpefalse
}
\makeatother

\date{}

\begin{document}

\centerline{\bf J.R.M. Antalan Research Paper}

\centerline{\bf May 28, 2019}

\centerline{} 

\centerline{} 

\centerline {\Large{\bf A Recreational Application of Two Integer Sequences}} 
\vspace{2mm}


\centerline{\Large {\bf and the Generalized Repetitious Number Puzzle}} 








\newtheorem{theorem}{Theorem}[section]
\newtheorem{lemma}[theorem]{Lemma}
\newtheorem{corollary}[theorem]{Corollary}
\newtheorem{proposition}[theorem]{Proposition}
\newtheorem{definition}[theorem]{Definition} 
\newtheorem{example}[theorem]{Example}
\newtheorem{remark}[theorem]{Remark}
\newtheorem{illustration}[theorem]{Illustration}
\newtheorem{puzzle}[theorem]{Puzzle}


\begin{abstract} 
In this article, we give a particular recreational application of the sequence A000533 and A261544 in ``The On-line Encyclopedia of Integer Sequences" (OEIS). The recreational application provides a direct extension to ``The Repetitious Number" puzzle of Martin Gardner contained in The Second Scientific American Book of Mathematical Puzzles and Diversions published in 1961. We then provide a generalization to the repetitious number puzzle and give a related puzzle as an illustrative example. Finally, as a consequence of the generalization, we define a family of sequence in which the sequences A000533 and  A261544 belong.           
\end{abstract} 
{\small{{\bf Keywords:} Integer Sequence$\cdot$ Repetitious Number$\cdot$ Generator$\cdot$ Length$\cdot$ Replication Number$\cdot$ $(l,r)$ co-divisor$\cdot$ $(l,r)$ co-divisor number$\cdot$ $(l,r)$ co-divisor sequence.\\
   {\bf AMS Classification Numbers:} 00A08, 11A05}}
\hrule
\vspace{4mm}

\begin{flushleft}
{\footnotesize{{\bf Author Information:}
\underline{John Rafael M. Antalan} 
\bigskip

(Assistant Professor, Department of Mathematics and Physics, College of Arts and Sciences, Central Luzon State University, 3120), Science City of Mu\~{n}oz, Nueva Ecija, Philippines.
\bigskip

e-mail: jrantalan@clsu.edu.ph   
}}
\end{flushleft}

\newpage
\section{Introduction}

We begin this section by giving a brief introductory information on the {\it On-line Encyclopedia of Integer Sequences} or OEIS. We then discuss the two integer sequences under study. Finally, we present the {\it Repetitious Number Puzzle} of Gardner that will serve as the ``source" of the recreational application.     

\subsection{The OEIS}

The OEIS (available at {\color{blue}{\it https://oeis.org/}}) is an on-line collection of over quarter-million number sequences initiated by Neil J.A. Sloane in early 1964 \cite{Sloane}. 
\bigskip

OEIS aims are (based on \cite{Sloane2}): 
\begin{enumerate}
\item To allow mathematicians or other scientists to find out if some sequence that turns up in their research has ever been seen before. If it has, they may find that the problem they're working on has already been solved, or partially solved, by someone else. 
Or they may find that the sequence showed up in some other situation, which may show them an unexpected relationship between their problem and something else. 
\item To have an easily accessible database of important, but difficult to compute, sequences.
\end{enumerate}

We illustrate the first aim using the paper of Rabago and Tagle in \cite{Rabago}. Their paper aims to find the integral dimensions of a rectangular prism (i.e. length, width and height) in which the surface area and the volume are numerically equal. The solution to their problem written in lexicographic order is surprisingly the sequence A229941 \cite{Lex} in the OEIS which gives a way for three regular polygons to snugly meet at a point.      
\\

For the second aim, a particular example of important but difficult to compute sequence in the OEIS is the sequence of {\it Mersenne primes} \cite{Sloane3}. The existence of which is equivalent to the existence of an even {\it perfect number} and the largest known {\it prime number}\ \cite{Gimps}.

\subsection{Two Integer Sequences in the OEIS}

We now turn our attention to two particular sequence in the OEIS. They are the sequences A000533 and A261544.  
\bigskip

The sequence A000533 \cite{Sloane4} in the OEIS is the sequence defined by   
\begin{align*}
\begin{split}
a(0)&=1\\
a(n)&=10^{n}+1, \hspace{5mm} n\geq 1.
\end{split}
\end{align*}

Its first 15 terms are:

\begin{center}
1, 11, 101, 1001, 10001, 100001, 1000001, 10000001, 100000001, 1000000001, 10000000001, 100000000001, 1000000000001, 10000000000001, 100000000000001,$\ldots$ . 
\end{center}

Daniel Arribas verified that ``$a(1)=11$ and $a(2)=101$ are the only prime terms of the sequence up to $n=100\ 000$" \cite{Sloane4}. Also, it is unknown whether there are other prime terms in the sequence.
\bigskip

On the otherhand, the sequence A261544 \cite{Gutko} is the sequence defined by 
\begin{equation*}
b(n)=\sum_{k=0}^n{1000^k}.
\end{equation*}  

Its first 10 terms are:

\begin{center}
1, 1001, 1001001, 1001001001, 1001001001001, 1001001001001001, 1001001001001001001, 1001001001001001001001, 1001001001001001001001001, 1001001001001001001001001001,$\ldots$ .
\end{center}

It can be verified that unlike the first sequence, the terms of this sequence are all composite except for the zeroth term ``1". A complete solution to this claim may be viewed in \cite{Gardiner}.    
\bigskip

With the two sequences in the OEIS already introduced, we are now ready to consider the {\it Repetitious Number Puzzle}. 

\subsection{The Repetitious Number Puzzle}

In \cite{Gardner}, Martin Gardner presented the puzzle given below:\\

\begin{displayquote} `` {\bf The Repetitious Number}. An unusual parlor trick is performed as follows. Ask spectator A to jot down any three-digit number, and then to repeat the digits in the same order to make a six-digit number (e.g., $394\ 394$). With your back turned so that you cannot see the number, ask A to pass the sheet of paper to spectator B, who is requested to divide the number by 7.

 Dont worry about the remainder, you tell him, because there won't be any. B is surprised to discover that you are right (e.g., $394\ 394$ divided by 7 is $56\ 342$). Without telling you the result, he passes it on to spectator C, who is told to divide it by 11. Once again you state that there will be no remainder, and this also proves correct ($56\ 342$ divided by 11 is $5\ 122$).

With your back still turned, and no knowledge whatever of the figures obtained by these computations, you direct a fourth spectator D, to divide the last result by 13. Again the division comes out even ($5\ 122$ divided by 13 is 394). This final result is written on a slip of paper which is folded and handed to you. Without opening it you pass it on to spectator A.

Open this, you tell him, and you will find your original three-digit number.

Prove that the trick cannot fail to work regardless of the digits chosen by the first spectator."\end{displayquote}

This puzzle was originally written by Yakov Perelman in \cite{Perelman}.\\  

In section 3, we discuss the solution of the puzzle and state some important questions necessary for its extension. At the moment, we discuss some important notations and mathematical concepts needed in understanding the solution of the puzzle and its generalization in the next section.    

\section{Preliminaries}

\subsection{Some Terms and Notations}

The following terms will be encountered in the succeeding sections of this article.

\begin{definition}
Let $n=d_1d_2\ldots d_kd_1d_2\ldots d_k\ \ldots\ d_1d_2\ldots d_k$ be a positive repetitive integer. We say that the positive integer $g=d_1d_2\ldots d_k$ is a {\bf generator} of $n$ if $g$ is a positive integer such that replicating $g$ a finite number of times generates $n$. 
\end{definition}

\begin{definition}
Let $g=d_1d_2\ldots d_k$ be a generator of $n$. Then the {\bf length} of $g$ denoted by $l(g)$ is the number of digits in $g$.     
\end{definition}

\begin{definition}
Let $g=d_1d_2\ldots d_k$ be a generator of $n$. The {\bf replication number} of $g$ denoted by $r(g)$ is the number of replication performed in $g$ in order to generate $n$.    
\end{definition}

To fully understand the concepts being discussed, we consider some examples.

\begin{example}   
Consider the positive repetitive integer $n_1=394\ 394$ in the {\it Repetitious Number Puzzle}. The positive integer $g_1=394$ is a generator for $n_1$ with {\it length} $l(g_1)=3$ and {\it replication number} $r(g_1)=2$. 
\end{example}

\begin{example}
The positive repetitive integer $n_2=111\ 111$ is generated by $g_2=1$ with {\it length} $l(g_2)=1$ and {\it replication number} $r(g_2)=6$. The integers 11, 111, and 111\ 111 are the other generators of $n_2$.  
\end{example}
 
\begin{example}
The positive integer $n_3=223\ 344$ generates itself with {\it length} $6$ and {\it replication number} $1$.     
\end{example}

\begin{remark}
We emphasize that a generator is not unique. Also, for any positive integer $n$, $n$ is a generator of itself. Moreover, if $n$ is not repetitive, then its generator is unique.
\end{remark}

\begin{remark}
 If $g$ generates $n$ with replication number $r$, we write $n=g_r$.   
\end{remark}

\subsection{Essential Mathematical Concepts}

For completeness, we recall some essential concepts in elementary Number Theory.

\begin{definition}
Let $a$ and $b$ be two positive integers such that $a\leq b$. We say that $a$ {\bf divides} $b$ written in symbol by $a\mid b$ if there is a positive integer $c$ such that \begin{equation} b=ac.\label{e1}\end{equation} If there is no positive integer $c$ that satisfies equation (\ref{e1}), then we say that $a$ {\bf does not divides} $b$ and this situation is denoted by $a\nmid b$. If $a\mid b$, we can also say the following:
\bigskip

 (i) $b$ is a multiple of $a$, (ii) $a$ is a divisor of $b$  and (iii) $a$ is a factor of $b$.    
\end{definition}

\begin{example}
Let us consider the positive integer $394\ 394$. Note that $7$ divides $394\ 394$ since 
\begin{equation*}
394\ 394=7\times 56\ 342.
\end{equation*} 
However, $5$ does not divides  $394\ 394$ since we cannot find any positive integer $c$ that can satisfy the equation
\begin{equation*}
394\ 394=5\times c.
\end{equation*}    
\end{example}

The property of divisibility given below is important.

\begin{lemma}
Let $a,b$ and $D$ be positive integers such that $D\leq a$ and $D\leq b$. If $D\mid a$ and $D\mid b$, then $D\mid (ax+by)$ for any positive integers $x$ and $y$.
\end{lemma}

The proof of \underline{Lemma 2.11} follows directly from the definition of divisibility and is standard in any elementary Number Theory textbooks. Letting $x=y=1$ we arrive at the the corollary given below.

\begin{corollary}
Let $a,b$ and $D$ be positive integers such that $D\leq a$ and $D\leq b$. If $D\mid a$ and $D\mid b$, then $D\mid (a+b)$.
\end{corollary}

\begin{remark}
The property of divisibility stated in \underline{Corollary 2.12} can be easily extended into a finite number of multiples. Given $D\mid a$ and $D\mid b$, by \underline{Corollary 2.12} we have $D\mid(a+b)$. Now, if $D\mid c$ given $D\mid (a+b)$ applying \underline{Corollary 2.12} once more gives $D\mid((a+b)+c)$ or $D\mid(a+b+c)$.   
\end{remark} 

If $b$ is divided by $a$ then either $a\mid b$ or $a\nmid b$. In both cases however, we may write $b$ in terms of $a$; this is guaranteed by the next theorem.

\begin{theorem}[{\bf Division Algorithm}]
Given integers $a$ and $b$ with $a>0$, there are unique integers $q$ and $r$ satisfying
\begin{equation*}
b=qa+r, \ \ \ \ \ \ \ \ \ \ 0\leq r<a.
\end{equation*}
\end{theorem} 

\begin{remark}
The integers $q$ and $r$ are respectively called the {\bf quotient} and the {\bf remainder} of $b$ upon division by $a$. Note also that $a\mid b$ if and only if $r=0$, and that $a\nmid b$ if and only if otherwise. 
\end{remark}

\begin{definition}
A positive integer $p>1$ is said to be a {\bf prime number} if its only positive divisors are $1$ and  $p$ itself. 
\end{definition}

\begin{example}
The positive integers $7,11$ and $13$ are all prime numbers; since their only positive divisors are $1$ and their selves.While the positive integer $394\ 394$ is not a prime number; since from \underline{Example 2.10}, we know that not only the positive integers $1$ and $394\ 394$ divides $394\ 394$ but also the integer $7$.    
\end{example}

\begin{theorem}[{\bf  Fundamental Theorem of Arithmetic}]
Every positive integer $n>1$ is either a prime or a product of primes; this representation is unique, apart from the order in which the factors occur.  
\end{theorem}

\underline{Theorem 2.18} is a well known result in elementary Number Theory, its proof is included in most of elementary Number Theory textbooks. Consider Burton in \cite{Burton} for instance.  

\begin{illustration}
Let us consider the positive integer $1\ 001$. From \underline{Theorem 2.18}, either $1\ 001$ is a prime or a product of primes. The latter holds true since \begin{equation*}1001=7\times 11\times 13.\end{equation*}
\end{illustration}

After a brief recall in some essential results in Elementary Number Theory, we are now ready to present the solution of the ``Repetitious Number Puzzle".

\subsection{Solution of the Repetitious Number Puzzle}

The solution discussed in this subsection is due to the solution presented by Gardner in \cite{Gardner}.\\

Any three digit number takes the form $d_1d_2d_3$ where $d_1,d_2$ and $d_3$ are non-negative integers  with bounds
\begin{equation*}
0<d_1\leq 9
\end{equation*}
\begin{equation*}
0\leq d_2\leq 9
\end{equation*}
\begin{equation*}
0\leq d_3\leq 9
\end{equation*}

Repeating the digits in the same order yields the six-digit integer $d_1d_2d_3d_1d_2d_3$. This integer can be factored into $1\ 001\times d_1d_2d_3$ as shown in the following computation

\begin{center}
\begin{tabular}{c@{\,}c@{\,}c@{\,}c@{\,}c@{\,}c@{\,}c@{\,}c}
&&&&$1$&$0$&$0$&$1$\\
$\times$ &&&&&$d_1$&$d_2$&$d_3$\\
\hline
&&&&$d_3$&$0$&$0$&$d_3$\\
&&&$d_2$&$0$&$0$&$d_2$&\\
$+$&&$d_1$&$0$&$0$&$d_1$&&\\
\hline
&&$d_1$&$d_2$&$d_3$&$d_1$&$d_2$&$d_3$.
\end{tabular}
\end{center}

Thus, $d_1d_2d_3$ and $1\ 001$ divides  $d_1d_2d_3d_1d_2d_3$ and that  $d_1d_2d_3d_1d_2d_3=1\ 001\times d_1d_2d_3$. From \underline{Illustration 2.19}, $1\ 001$ can be expressed as a product of primes $7,11$ and $13$. Hence $7,11$ and $13$ divides $d_1d_2d_3d_1d_2d_3$ and that  
\begin{equation*}
d_1d_2d_3d_1d_2d_3=7\times 11\times 13\times d_1d_2d_3.
\end{equation*}

In lieu of \underline{Theorem 2.14}, we have 
\begin{equation*}
{d_1d_2d_3d_1d_2d_3}=7\times (11\times 13\times d_1d_2d_3)+0.
\end{equation*}
So, dividing the integer $d_1d_2d_3d_1d_2d_3$ by $7$ gives the integer $11\times 13\times d_1d_2d_3$ with remainder $0$.\\

Next, we consider the integer $11\times 13\times d_1d_2d_3$. In lieu of \underline{Theorem 2.14}, we have  
\begin{equation*}
11\times 13\times d_1d_2d_3=11\times (13\times d_1d_2d_3)+0.
\end{equation*}
So, dividing the integer $11\times 13\times d_1d_2d_3$ by $11$ gives the integer $13\times d_1d_2d_3$ with remainder $0$.\\

Finally, we consider the integer  $13\times d_1d_2d_3$ . Note that this integer can be written as

\begin{equation*}
13\times d_1d_2d_3=13\times (d_1d_2d_3)+0.
\end{equation*}
So, dividing the integer $13\times d_1d_2d_3$ by $13$ gives the integer $ d_1d_2d_3$ with remainder $0$.\\

Hence, dividing the six-digit repetitive number $d_1d_2d_3d_1d_2d_3$ in succession by the integers $7,11$ and $13$ returns the repetitive number into its generator $d_1d_2d_3$. This solves the puzzle.

\begin{remark}
The order of dividing the integer $d_1d_2d_3d_1d_2d_3$ by the integers $7,11$ and $13$ do not matter in the puzzle. For $d_1d_2d_3d_1d_2d_3$ can be written as \begin{center} $7\times (11\times 13\times d_1d_2d_3)$, $7\times (13\times 11\times d_1d_2d_3)$ $11\times (7\times 13\times d_1d_2d_3)$, $11\times (13\times 7\times d_1d_2d_3)$, $13\times (7\times 11\times d_1d_2d_3)$ and $13\times (11\times 7\times d_1d_2d_3)$. \end{center}    
\end{remark}

We now present our results in the next section. 

\section{Results}

\subsection{Recreational Application of the Sequence A000533}

The goal of this subsection is to show that the $k^{th}$ term $a(k)$ of the sequence  A000533 divides the $2k-digit$ repetitive number $n$  generated by $g$ with $l(g)=k$. Hence, when a $k-digit$  number $g$ is duplicated resulting to $n$, dividing $n$ with the prime factors of $a(k)$ gives the original number $g$. This result is due to
\begin{theorem}
Let $n=(d_1d_2\ldots d_k)_2$ be a repetitive number generated by $g=d_1d_2\ldots d_k$ of length $k$. Then there is a finite sequence of divisors $D_i$ such that $n$ upon division by all  of $D_i$ becomes $g$.   
\end{theorem}

\begin{proof}
Given a repetitive number $n=(d_1d_2\ldots d_k)_2$, we express it as a sum of two positive integers both divisible by $g=d_1d_2\ldots d_k$. In particular $n$ can be expressed as the sum

\begin{center}
\begin{tabular}{c@{\,}c@{\,}c@{\,}c@{\,}c@{\,}c@{\,}c@{\,}c@{\,}c}
&$d_1$&$d_2$&$\ldots$&$d_k$&$0$&$0$&$\ldots$&$0$\\
$+$ &&&&&$d_1$&$d_2$&$\ldots$&$d_k$\\
\hline
&$d_1$&$d_2$&$\ldots$&$d_k$&$d_1$&$d_2$&$\ldots$&$d_k$.
\end{tabular}
\end{center}
Note that since $g\mid g$ and $g\mid d_1d_2\ldots d_k\underbrace{00\ldots 0}_\text{k-zeros}$, by \underline{Corollary 2.12} we have \begin{center}$g\mid  (g+d_1d_2\ldots d_k\underbrace{00\ldots 0}_\text{k-zeros})$.\end{center} But $g+d_1d_2\ldots d_k\underbrace{00\ldots 0}_\text{k-zeros}=n$. So, $g\mid n$.\\

After factoring out the common factor $g$ in both summands we have
\begin{align*}
 n&=g\times (1+1\underbrace{00\ldots 0}_\text{k-zeros})\\
&=g\times (1\underbrace{00\ldots 0}_\text{$k-1$-zeros}1)\\
&=g\times a(k).
\end{align*}   

By the Fundamental Theorem of Arithmetic (\underline{Theorem 2.18}), $a(k)$ is either a prime or a product of primes. If $a(k)$ is prime, then the finite sequence of divisors to be divided to $n$ to become $g$ is $a(k)$ itself. If $a(k)$ is non-prime then the finite sequence of divisors to be divided to $n$ to become $g$ is the finite sequence whose terms are the prime divisors of  $a(k)$.       
\end{proof}

The proof of \underline{Theorem 3.1} gives us a method on solving a particular extension of the {\it repetitious number puzzle}.\\
 
{\bf Problem.} Suppose that in the {\it repetitious number puzzle} spectator A was asked to write down any $k-$digit positive integer. To what sequence of prime numbers does the resulting $2k-$digit number be divided in order to return to the original $k-$digit number?\\    

{\bf Solution.} Let $g$ be the $k-$digit positive integer and let $n$ be the resulting $2k-$digit number. Note that $n$ is a repetitive number generated by $g$ of length $k$ with replication number $2$. That is, $n=g_2$ with $l(g)=k$. Using the result contained  in the proof of \underline{Theorem 3.1}, we must divide $n$ by the prime divisors of  $1\underbrace{00\ldots 0}_\text{k-1-zeros}1$, or the $k^{th}$ term of the sequence A000533 in order to return to the original number $g$. 

\begin{illustration}
Suppose that spectator $A$ wrote the number $g=451\ 220\ 125$. Duplicating $g$ gives the number $n=451\ 220\ 125\ 451\ 220\ 125$. Dividing $n$ by the numbers $7,11,13,19$ and $52\ 579$, the prime divisors of  $1\ 000\ 000\ 001$ which is the $9^{th}$ term of the sequence A000533 (See Table 1), gives the original number $g=451\ 220\ 125$.   
\end{illustration}

\begin{sidewaystable}
\begin{center}
\captionof{table}{\bf Prime factorization of the first 25 terms of the sequence A000533 (Generated using Wolfram Alpha\cite{WA})}
\begin{tabular}{ |c|c|c|c| }
\hline
No. of Digits ($k$) & Rep. No. ($r$) & Terms of Sequence A000533 & Prime Factorization\\
\hline
0 & 2 & 1 & -\\
\hline
1 & 2 & 11 & 11\\
\hline
2 & 2 & 101 & 101\\
\hline
3 & 2 & 1001 & $7\cdot 11\cdot 13$\\
\hline
4 & 2 & 10001 & $73\cdot137$\\
\hline
5 & 2 & 100001 & $11\cdot 9091$\\
\hline
6 & 2 & 1000001 & $101\cdot 9901$\\
\hline
7 & 2 & 10000001 & $11\cdot 909091$\\
\hline
8 & 2 & 100000001 & $17\cdot 5882353$\\
\hline
9 & 2 & 1000000001 & $7\cdot 11\cdot 13\cdot 19\cdot 52579$\\
\hline
10 & 2 & 10000000001 & $101\cdot 3541\cdot 27961$\\
\hline
11 & 2 & 100000000001 & $11\cdot 23\cdot 4093\cdot 8779$\\
\hline
12 & 2 & 1000000000001 & $73\cdot 137\cdot 99990001$\\
\hline
13 & 2 & 10000000000001 & $11\cdot 859\cdot 1058313049$\\
\hline
14 & 2 & 100000000000001 & $29\cdot 101\cdot 281\cdot 121499449$\\
\hline
15 & 2 & 1000000000000001 & $7\cdot 11\cdot 13\cdot 211\cdot 241\cdot 2161\cdot 9091$\\
\hline
16 & 2 & 10000000000000001 & $353\cdot 449\cdot 641\cdot 1409\cdot 69857$\\
\hline
17 & 2 & 100000000000000001 & $11\cdot 103\cdot 4013\cdot 21993833369$\\
\hline
18 & 2 & 1000000000000000001 & $101\cdot 9901\cdot 999999000001$\\
\hline
19 & 2 & 10000000000000000001 & $11\cdot 909090909090909091$\\
\hline
20 & 2 & 100000000000000000001 & $73\cdot 137\cdot 1676321\cdot 5964848081$\\
\hline
21 & 2 & 1000000000000000000001 & $7\cdot 11\cdot 13\cdot 127\cdot 2689\cdot 459691\cdot 909091$\\
\hline
22 & 2 & 10000000000000000000001 & $89\cdot 101\cdot 1052788969\cdot 1056689261$\\
\hline
23 & 2 & 100000000000000000000001 & $11\cdot 47\cdot 139\cdot 2531\cdot 549797184491917$\\
\hline
24 & 2 & 1000000000000000000000001 & $17\cdot 5882353\cdot 9999999900000001$\\
\hline
25 & 2& 10000000000000000000000001 & $11\cdot 251\cdot 5051\cdot 9091\cdot 78875943472201$\\
\hline
\end{tabular}
\end{center}
\end{sidewaystable}

\subsection{Recreational Application of the Sequence A261544}

The goal of this subsection is to show that the $(r-1)^{st}$ term $b(r-1)$ of the sequence  A261544 divides the $3r-digit$ repetitive number $n$  generated by $g$ with $l(g)=3$. Hence, when a $3-digit$  number $g$ is replicated $r$-times resulting to $n$, dividing $n$ with the prime factors of $b(r-1)$ gives the original number $g$. This result is due to

\begin{theorem}
Let $n=(d_1d_2d_3)_r$ be a repetitive number generated by $g=d_1d_2d_3$ of length $3$. Then there is a finite sequence of divisors $D_i$ such that $n$ upon division by all  of $D_i$ becomes $g$.   
\end{theorem}

\begin{proof}
Given a repetitive number $n=(d_1d_2d_3)_r$, we express it as a sum of $r$ positive integers both divisible by $g=d_1d_2d_3$. In particular $n$ can be expressed as the sum

\begin{equation*}
n=d_1d_2d_3(0)_{3(r-1)}+d_1d_2d_3(0)_{3(r-2)}+\ldots+d_1d_2d_3(0)_{3(r-r)}.
\end{equation*}

Note that since $g\mid g$ and $g\mid d_1d_2d_3(0)_{3j}$, for $j=1,2,\ldots r-1$, by \underline{Corollary 2.12} we have \begin{center}$g\mid d_1d_2d_3(0)_{3(r-1)}+d_1d_2d_3(0)_{3(r-2)}+\ldots+d_1d_2d_3(0)_{3(r-r)}$.\end{center} So, $g\mid n$.\\

Factoring out the common factor $g$ in all of the summands we have
\begin{align*}
 n&=g\times \Big(1(0)_{3(r-1)}+1(0)_{3(r-2)}+\ldots+1(0)_{3(r-r)}\Big)\\
&=g\times b(r-1).
\end{align*}  

By the Fundamental Theorem of Arithmetic (\underline{Theorem 2.18}), $b(r-1)$ is either a prime or a product of primes. However, we know that (except for the zeroth term) the terms of the sequence A261544 are all composite. So the finite sequence of divisors to be divided to $n$ in order to become $g$ is the finite sequence whose terms are the prime divisors of  $b(r-1)$.       
\end{proof}

The proof of \underline{Theorem 3.2} gives us a method on solving another particular extension of the {\it repetitious number puzzle}.\\
 
{\bf Problem.} Suppose that in the {\it repetitious number puzzle} spectator A was asked to write down any $3-$digit positive integer and replicate it $r-$times. To what sequence of numbers does the resulting $3r-$digit number be divided in order to return to the original $3-$digit number?\\    

{\bf Solution.} Let $g$ be the $3-$digit positive integer and let $n$ be the resulting $3r-$digit number. Note that $n$ is a repetitive number generated by $g$ of length $3$ with replication number $r$. That is, $n=g_r$ with $l(g)=3$. Using the result contained  in the proof of Theorem 3.2, we must divide $n$ by the prime divisors of  $b(r-1)$, the $(r-1)^{st}$ term of the sequence A261544 in order to return to the original number $g$. 

\begin{illustration}
Suppose that spectator $A$ wrote the number $g=721$. Replicating $g$ 4-times gives the number $n=721\ 721\ 721\ 721$. Dividing $n$ by the numbers $7,11,13,101,9\ 091$  the prime divisors of the third term of the sequence A261544 which is $1\ 001\ 001\ 001$ (See Table 2), gives the original number $g=721$.   
\end{illustration}

\begin{sidewaystable}
\begin{center}
\captionof{table}{\bf Prime factorization of the first nine terms of the sequence A261544 (Generated using Wolfram Alpha \cite{WA})}
\begin{tabular}{ |c|c|c|c| }
\hline
Number of Digits ($k$) & Number of Repetitions ($r$) & Terms of the Sequence A261544 & Prime Factorization\\
\hline
3 & 1 & 1 & $-$\\
\hline
3 & 2 & 1001 & $7\cdot11\cdot13$\\
\hline
3 & 3 & 1001001 & $3\cdot333667$\\
\hline
3 & 4 & 1001001001 & $7\cdot11\cdot13\cdot101\cdot9091$\\
\hline
3 & 5 & 1001001001001 & $31\cdot 41\cdot 271\cdot 2906161$\\
\hline
3 & 6 & 1001001001001001 & $3\cdot7\cdot11\cdot13\cdot19\cdot52579\cdot333667$\\
\hline
3 & 7 & 1001001001001001001 & $43\cdot239\cdot1933\cdot4649\cdot10838689$\\
\hline
3 & 8 & 1001001001001001001001 & $7\cdot11\cdot13\cdot73\cdot101\cdot137\cdot9901\cdot99990001$\\
\hline
3 & 9 & 1001001001001001001001001 & $33\cdot757\cdot333667\cdot440334654777631$\\
\hline
\end{tabular}
\end{center}
\end{sidewaystable}

\subsection{Generalized Repetitious Number Puzzle}

In this subsection, we generalize the {\it repetitious number puzzle} by allowing spectator A to write down any $k-$digit number and replicate it $r-$times to generate the integer $n=g_r$ with $l(g)=k$. The generalization is given in the next theorem.

\begin{theorem}
Let $n=(d_1d_2\ldots d_k)_r$ be a repetitive number generated by $g=d_1d_2\ldots d_k$ of length $k$. Then the sequence of prime factors of the integer  
\begin{equation*}
\Big(1(0)_{k-1}\Big)_{r-1}1
\end{equation*}
is a finite sequence such that $n$ upon division by all the sequence terms becomes $g$.   
\end{theorem}

\begin{proof}
Given a repetitive number $n=(d_1d_2\ldots d_k)_r$, we express it as a sum of $r$ positive integers both divisible by $g=d_1d_2\ldots d_k$. In particular $n$ can be expressed as the sum

\begin{equation*}
n=d_1d_2\ldots d_k(0)_{k(r-1)}+d_1d_2\ldots d_k(0)_{k(r-2)}+\ldots+d_1d_2\ldots d_k(0)_{k(r-r)}.
\end{equation*}

Note that since $g\mid g$ and $g\mid d_1d_2\ldots d_k(0)_{kj}$, for $j=1,2,\ldots r-1$, by \underline{Corollary 2.12} we have \begin{center}$g\mid d_1d_2\ldots d_k(0)_{k(r-1)}+d_1d_2\ldots d_k(0)_{k(r-2)}+\ldots+d_1d_2\ldots d_k(0)_{k(r-r)}$.\end{center} So, $g\mid n$.\\

Factoring out the common factor $g$ in all of the summands we have
\begin{align*}
 n&=g\times \Big(1(0)_{k(r-1)}+1(0)_{k(r-2)}+\ldots+1(0)_{k(r-r)}\Big)\\
&=g\times \Big(1(0)_{k-1}\Big)_{r-1}1.
\end{align*}  

By the Fundamental Theorem of Arithmetic (\underline{Theorem 2.18}), $\Big(1(0)_{k-1}\Big)_{r-1}1$ is either a prime or a product of primes. If $\Big(1(0)_{k-1}\Big)_{r-1}1$  is prime, then the finite sequence of divisors to be divided to $n$ to become $g$ is $\Big(1(0)_{k-1}\Big)_{r-1}1$ itself. If $\Big(1(0)_{k-1}\Big)_{r-1}1$  is non-prime then the finite sequence of divisors to be divided to $n$ to become $g$ is the finite sequence whose terms are the prime divisors of  $\Big(1(0)_{k-1}\Big)_{r-1}1$.        
\end{proof}

\underline{Theorem 3.5} proves the validity of a Grade 7 teacher's clever way in verifying if his students correctly performed a sequence of division.   

\begin{puzzle}
{\bf A Relay Involving Division of Large Numbers.} Sir DELTA is grade 7 mathematics teacher in the Philippines. To test the proficiency of his students on performing division of large numbers, he grouped his students such that each group is consists of 10 members. 
\\

He then instruct the first student which we name S1 to write down in a 1/4 sheet of paper any $4-$digit positive integer (say 2\ 019) and replicate it $8-$ times to get a $32-$digit number ($20\ 192\ 019\ 201\ 920\ 192\ 019\ 201\ 920\ 192\ 019$). Then he asked S1 to give the paper containing the $32-$digit number to S2. S2 then was asked to divide the $32-$digit number by $17$ and write down the answer (1\ 187\ 765\ 835\ 407\ 070\ 118\ 776\ 583\ 540\ 707) in another 1/4 sheet of paper. After S2 was done writing the answer in a 1/4 sheet of paper, Sir Delta asked S2 to give the paper to S3.
\bigskip

Denote by $A_n$ the answer of student n. Suppose that the process continues with the following given (See Table 3):
\begin{itemize}
\item S3 performs $A_2\div 73$
\item S4 performs $A_3\div 137$ 
\item S5 performs $A_4\div 353$
\item S6 performs $A_5\div 449$
\item S7 performs $A_6\div 641$
\item S8 performs $A_7\div 1\ 409$
\item S9 performs $A_8\div 69\ 857$
\item S10 performs $A_9\div 5\ 882\ 353$.
\end{itemize}          

Sir DELTA then asked S10 to give his/her answer to him.
\\ 

If Sir DELTA wants to determine whether his students performed their assigned division problem correctly or not, prove that it is enough for him to ask S1: ``Is this your $4-$digit number?"   
\end{puzzle}

\begin{sidewaystable}
\begin{center}
\captionof{table}{\bf  Some integers of the form $\big(1(0)_{k-1}\big)_{r-1}1$ with their corresponding prime divisors (Generated using Wolfram Alpha \cite{WA}).}
\begin{tabular}{ |c|c|c|c| }
\hline
($j$) & ($r$) & $\big(1(0)_{j-1}\big)_{r-1}1$ & Prime Divisors\\
\hline
1 & 1 & 1 & -\\
\hline
2 & 10 & 1010101010101010101 & (41),(101),(271),(3541),(9091),(27961)\\
\hline
3 & 9 & 1001001001001001001001001 & (3),(3),(757),(333667),(440334654777631)\\
\hline
4 & 8 & 10001000100010001000100010001 & (17),(73),(137),(353),(449),(641),(1409),(69857),(5882353)\\
\hline
5 & 7 & 1000010000100001000010000100001  & (71),(239),(4649),(123551),(102598800232111471)\\
\hline
6 & 6 & 1000001000001000001000001000001 & (3),(19),(101),(9901),(52579),(333667),(999999000001)\\
\hline
7 & 5 & 10000001000000100000010000001 & (41),(71),(271),(123551),(102598800232111471)\\
\hline
8 & 4 & 1000000010000000100000001 & (17),(353),(449),(641),(1409),(69857),(5882353)\\
\hline
9 & 3 & 1000000001000000001 & (3),(757),(440334654777631)\\
\hline
10 & 2 & 10000000001 & (101),(3541),(27961)\\
\hline
\end{tabular}
\end{center}
\end{sidewaystable}

\begin{illustration}
Given below are the correct answers for the assigned sequence of division in \underline{Puzzle 3.6} generated using Wolfram Alpha \cite{WA}.

\small
\begin{equation*}
20\ 192\ 019\ 201\ 920\ 192\ 019\ 201\ 920\ 192\ 019\div 17=1\ 187\ 765\ 835\ 407\ 070\ 118\ 776\ 583\ 540\ 707
\end{equation*}

\begin{equation*}
1\ 187\ 765\ 835\ 407\ 070\ 118\ 776\ 583\ 540\ 707\div 73=16\ 270\ 764\ 868\ 590\ 001\ 627\ 076\ 486\ 859 
\end{equation*}

\begin{equation*}
16\ 270\ 764\ 868\ 590\ 001\ 627\ 076\ 486\ 859\div 137=118\ 764\ 707\ 070\ 000\ 011\ 876\ 470\ 707 
\end{equation*}

\begin{equation*}
118\ 764\ 707\ 070\ 000\ 011\ 876\ 470\ 707\div 353=336\ 443\ 929\ 376\ 770\ 571\ 888\ 019  
\end{equation*}

\begin{equation*}
336\ 443\ 929\ 376\ 770\ 571\ 888\ 019\div 449=749\ 318\ 328\ 233\ 342\ 030\ 931  
\end{equation*}

\begin{equation*}
749\ 318\ 328\ 233\ 342\ 030\ 931\div 641=1\ 168\ 983\ 351\ 378\ 068\ 691  
\end{equation*}

\begin{equation*}
1\ 168\ 983\ 351\ 378\ 068\ 691\div 1\ 409=829\ 654\ 614\ 178\ 899  
\end{equation*}

\begin{equation*}
829\ 654\ 614\ 178\ 899\div 69\ 857=11\ 876\ 470\ 707  
\end{equation*}

\begin{equation*}
11\ 876\ 470\ 707\div 5\ 882\ 353=2\ 019.  
\end{equation*}
\normalsize 
\end{illustration}

\subsection{The $(l,r)$ co-divisor Number and the Family of $(l,r)$ co-divisor Sequences}

We learned from the previous subsection that given a repetitive number  $n=(d_1d_2\ldots d_k)_r$ that is generated by $g=d_1d_2\ldots d_k$ of length $k$, we have

\begin{equation}
n=g\times \Big(1(0)_{k-1}\Big)_{r-1}1.
\label{jr}    
\end{equation}

The number $\Big(1(0)_{k-1}\Big)_{r-1}1$ due to its importance will be named and defined formally in the next definition.

\begin{definition}
Let $k,r\in \mathbb{Z}^+$. The number $\Big(1(0)_{k-1}\Big)_{r-1}1$ in equation (\ref{jr}) is called the {\bf $(l,r)$ co-divisor of $g$ relative to $n$}. 
\end{definition}   

\begin{remark}
The name $(l,r)$ co-divisor number defined on \underline{Definition 3.8} is base from the idea that the number $\Big(1(0)_{k-1}\Big)_{r-1}1$ is dependent to the length $(l)$ of the generator and its replication number $(r)$.   
\end{remark}

\begin{example}
Recall that in the {\it Repetitious Number Puzzle} we have 
\begin{equation*}
394\ 394=394\times 1\ 001.
\end{equation*}
Hence, the $(3,2)$ co-divisor of $394$ relative to $394\ 394$ is $1\ 001$. In general, given any positive integer $g$ of length $3$ when duplicated has the $(l,r)$ co-divisor of $1\ 001$.     
\end{example}

\begin{remark}
To avoid redundancy, we drop the word ``relative to $n$" in determining the $(l,r)$ co-divisor of $g$. This is because the $(l,r)$ co-divisor of an integer $g$ is completely determined by the length of $g$ which is $k$ and the number of replications performed in $g$ which is $r$. 
\end{remark}

\begin{example}
In \underline{Illustration 3.2}, the $(9,2)$ co-divisor of $451\ 220\ 125$ is  $1\ 000\ 000\ 001$. In general, any positive integer $g$ of length $9$ when duplicated has the $(l,r)$ codivisor of $1\ 000\ 000\ 001$.    
\end{example}

\begin{example}
Let $g$ be a $3$-digit positive integer. The $(3,4)$ codivisor of $g$ is $1\ 001\ 001\ 001$. (See \underline{Illustration 3.4})  
\end{example}

\begin{example}
Let $g$ be a $4$-digit positive integer. The $(4,8)$ codivisor of $g$ is the number 
$10\ 001\ 000\ 100\ 010\ 001\ 000\ 100\ 010\ 001$. (See \underline{Puzzle 3.6})  
\end{example}

The concept of $(l,r)$ co-divisor allows us to view the sequence A000533 and the sequence A261544 in the OEIS as a particular member of a family of sequence which we call {\bf $(l,r)$ co-divisor sequences}.  
\bigskip

In particular, if we let $s(k,r)= \big(1(0)_{k-1}\big)_{r-1}1$ we have
\begin{equation*}
s(k,2)=a(k), j=1,2,3,\ldots
\end{equation*}
 where $a(k)$ is the $k^{th}$ term of the sequence  A000533. 
We also have
\begin{equation*}
s(3,r)=b(r-1), r=1,2,3,\ldots.
\end{equation*}
where $b(r-1)$ is the $(r-1)^{st}$ term of the sequence  A261544. 
\bigskip

We end this paper by recommending further studies on the $(l,r)$ co-divisor number and the $(l,r)$ co-divisor sequences and their applications.

\section{Conclusion}

In this article, we discussed the {\it Repetitious Number Puzzle} and its solution. We established that the {\it Repetitious Number Puzzle} is equivalent to the problem: 

\begin{center}
{\it Given a positive integer generator $g$ of length $k$ that is to be replicated $r-$times resulting to the integer $n$ of length $kr$, by what prime numbers must $n$ be divided such that upon dividing $n$ by all of the prime numbers gives back $g$?}      
\end{center}

where the length of $g$ is $3$ and the number of replication $r$ is 2. 
\bigskip

We then provide a generalization to the puzzle by first taking $k\geq 3$. We showed that the solution to the puzzle when $k\geq 3$ is given by the prime divisors of $a(k)$ where $a(k)$ is the $k^{th}$ term of the sequence A000533. Then fixing $k=3$, we consider $r\geq 2$. In this case, we showed that the solution to the puzzle when $k=3$ and $r\geq 2$ is given by the prime divisors of $b(r-1)$ where $b(r-1)$ is the $(r-1)^{st}$ term of the sequence A261544.
\bigskip

For the general case where $k\geq 3$ and $r\geq 2$, we showed that the solution to the puzzle is given by the prime divisors of the $(l,r)$ co-divisor number $\Big(1(0)_{k-1}\Big)_{r-1}1$. 
The concept of $(l,r)$ co-divisor number allowed the possibility to view the sequence A000533 and the sequence A261544 in the OEIS as a particular member of a family of sequences which we call $(l,r)$ co-divisor sequences. Further studies on on the $(l,r)$ co-divisor number and the $(l,r)$ co-divisor sequences and their applications are then recommended.  

\section{Acknowledgment}

The author is thankful to every individual and organizations who are responsible for the creation of this research article. The creation of this article would not be possible without the suggestion of Mr. Melchor A. Cupatan of the Department of Mathematics and Physics Central Luzon State University. The author would also like to express his sincere gratitude to Ms. Josephine Joy V. Tolentino of Philippine Science High School Central Luzon Campus for her motivation and valuable comments  that lead to the improvement of the manuscript. Also, the author thanks the Central Luzon State University for their untarnished effort in encouraging its faculty members to do research. Finally, the author would like to thank the various referees for their valuable comments and suggestions that helps improve the content of the paper.


\end{document}